\documentclass[12pt]{amsart} \usepackage{amscd}
\usepackage[all]{xy}
\usepackage{amssymb}
\UseComputerModernTips

\newtheorem{theorem}{Theorem}
\newtheorem{lemma}[theorem]{Lemma}
\newtheorem{proposition}[theorem]{Proposition}

\theoremstyle{definition}                 
\newtheorem{example}{Example}            
\newtheorem{remark}{Remark} \newtheorem*{notation}{Notation}

\usepackage{amsfonts}
\newcommand{\field}[1]{\mathbb{#1}}          \newcommand{\Q}{\field{Q}}
\newcommand{\R}{\field{R}}                   \newcommand{\Z}{\field{Z}}

                    \newcommand{\e}{\varepsilon}

\newcommand{\fg}{\mathfrak     g}     \newcommand{\fp}{\mathfrak    p}
 \newcommand{\fh}{\mathfrak h}
\newcommand{\fu}{\mathfrak u}

 \newcommand{\ra}{\rightarrow}

\begin{document}

\title[Hypergeometric Groups of Orthogonal Type]
{Hypergeometric Groups of Orthogonal Type}

\author{T.N.Venkataramana}

{\address{ T.N.Venkataramana, School of Mathematics, TIFR, Homi Bhabha
Road, Colaba, Mumbai 400005, India}

\email{venky@math.tifr.res.in}

\subjclass{primary:  22E40.  Secondary:  20F36  \\  T.N.Venkataramana,
School of  Mathematics, Tata  Institute of Fundamental  Research, Homi
Bhabha Road, Colaba, Mumbai 40005, INDIA}

\date{}

\begin{abstract} We obtain an infinite family of orthogonal hypergeometric 
groups which are higher rank arithmetic groups. We also list cases of 
arithmetic monodromy when the real Zariski closure of the 
hypergeometric group is $O(2,3)$. 
\end{abstract}

\maketitle{}

\section{Introduction}\label{introsection}

Consider the $_nF_{n-1}$ type hypergeometric differential equation 
\[D(\alpha,\beta,q)u=0\]  where $q$ varies  over the  thrice punctured
sphere $C={\mathbb P}^1\setminus  \{0,1,\infty\}$, $\alpha , \beta \in
\Q ^n$, 
\[D=D(\alpha,\beta,q)=  \prod  _{i=1}^n  (\theta  +\beta  _i-1)-q\prod
_{i=1}^n (\theta +\alpha _i),\] and $\theta =q\frac{d}{dq}$. Thus, $D$
and  $\theta $  are  viewed  as differential  operators  on the  curve
$C$. The  fundamental group of $C$  (namely the free group  on the two
generators which can be taken to be small loops around $0$ and $\infty
$)  acts on the  space of  solutions of  this equation  (the monodromy
representation).  The action is  by analytic continuation of solutions
along the  loops corresponding to  elements of the  fundamental group.
The image of the resulting representation is called the hypergeometric
group corresponding  to the  parameters $\alpha ,  \beta$.  It  is the
group generated  by $h_0,h_{\infty}$ where $h_0$  and $h_{\infty}$ are
the images  of the loops around  $0$ and $\infty$  under the monodromy
representation. \\

A  Theorem  of  Levelt  completely  describes  the  ($_nF_{n-1}$-type)
hypergeometric  monodromy  representation.    We  breifly  recall  the
description.  Assume that $\alpha _j-  \beta _k$ is not an integer for
any $j,k$.  Equivalently, if  we put $f(x)=\prod _{j=1}^n (x-e^{2\pi i
\alpha _i})$  and $g(x)=\prod _{j=1}^n (x-e^{2\pi i  \beta _j})$, then
$f,g$ have no common roots.  Denote by $A,B$ the companion matrices of
$f,g$ respectively.   Then Levelt's theorem  says that there  exists a
basis $\{u\}$  of solutions  of the foregoing  hypergeometric equation
$Du=0$ with respect to which, the matrix of the action of $h_0$ is $A$
and that  of $h_{\infty}$  is $B^{-1}$. Thus  a small loop  around $1$
goes  to  the matrix  $C=  A^{-1}B$  since $h_0h_1h_{\infty}=1$.   The
element  $C$  is  a  complex  reflection.   Moreover,  given  any  two
co-prime, monic,  degree $n$ polynomials $f,g$,  the representation of
the  free group  $F_2=<x_0,x_{\infty}>$  on two  generators, given  by
$x_0\mapsto  A$,  $x_{\infty}\mapsto  B^{-1}$  is  the  hypergeometric
monodromy representation corresponding  to parameters $\alpha , \beta$
with the  roots of $f,g$ being  the exponentials of  $\alpha _j, \beta
_k$ as before.  \\

Beukers  and  Heckman (\cite{Beu-Hec})  have  completely analysed  the
Zariski  closure  $G$ of  the  foregoing  monodromy.   We now  briefly
describe their result, making  the (simplifying) assumption that $f,g$
are products  of cyclotomic  polynomials; their coefficients  are then
integers.  Thus the hypergeometric  group is a subgroup of $GL_n(\Z)$.
Assume  also that $f,g$  form a  primitive pair  (\cite{Beu-Hec}).  In
\cite{Beu-Hec} it  is proved (see also \cite{FMS}, p.6 (2)) that If 
$f(0)/g(0)=-1$,  then the Zariski
closure $G$ is (even as an algebraic group over $\Q$) either finite or
the orthogonal  group $O_n(h)$ of a  non-degenerate rational quadratic
form $h$  (if $f(0)/g(0)=1$, and the other conditions are the same, 
then $n$  is even and  $G$ is the  symplectic group).
Thus  the hypergeometric  group is  either a  subgroup of  an integral
orthogonal group or a subgroup of the integral symplectic group. \\

In  \cite{Sar}, Sarnak  has  asked when  the  hypergeometric group  is
arithmetic  (i.e. has  finite  index  in the  integral  points of  its
Zariski  closure). Otherwise,  the  group  is said  to  be {\it  thin}
(\cite{Sar}).   In  \cite{FMS},  the   authors  prove  that  when  the
resulting  quadratic  form  has  signature  $(n-1,1)$,  (with  very  few
exceptions; see  Conjecture 2 of \cite{FMS})  the hypergeometric group
is often {\it thin} i.e. has infinite index in the integral orthogonal
group.  However,  when the  quadratic form has  higher rank  i.e.  has
signature  $(p,q)$ with  $p,q \geq  2$ or  when the  group $G$  is the
symplectic group, the  situation is less clear. In  \cite{SV}, the case
when  $G$ is  the  symplectic group  is  considered. It  is proved  in
\cite{SV} that in  a sizeable number of cases,  the monodromy group is
an  arithmetic group; however,  7 thin  examples (with  $G=Sp_4$) have
been given  by Brav and  Thomas \cite{BT} using a  ping-pong argument.
We understand that  the methods of \cite{BT} can  also prove thin-ness
in  some  higher  rank  orthogonal   cases.   An  example  of  a  thin
hypergeometric group in $O(2,2)$ is given in \cite{F}. \\

In the  present paper  we show that  for infinitely many  odd integers
$n$, and  for suitable  parameters $\alpha, \beta$  the hypergeometric
group is arithmetic, i.e. has  finite index in the integral orthogonal
group.   We  also give  many  examples  of  arithmetic monodromy  when
$G=O(2,3)$ over  $\R$ but has $\Q$  rank either one or  two.  There is
some interest in constructing these examples because these are perhaps
the first examples of  higher rank arithmetic hypergeometric monodromy
groups which are of {\it orthogonal} type. We prove 

\begin{theorem} \label{maintheorem}  Let $m\geq 0$ be  an integer. Let
$f_0(x)=x^5-1$ and  $g_0(x)=(x+1)(x^2+1)^2$.  Suppose $P,Q  \in \Z[x]$
are   co-prime  monic  polynomials   of  degree   $m$  such   that  if
$f(x)=f_0(x)P(x^6)$  and $g(x)=g_0(x)Q(x^6)$  then $f,g$  are co-prime
polynomials.  Then  the hypergeometric monodromy  group $\Gamma (f,g)$
is an  arithmetic subgroup of an integral  orthogonal group $O(h)(\Z)$
with $\Q-rank (h)\geq 2$.
\end{theorem}

For          example,          if         \[f=(x^5-1)(x^{12}+x^6+1)^m,
\quad g=(x+1)(x^2+1)^2(x^{12}+1)^m, or, \]  
\[f=(x^5-1)(x^{12}-x^6+1)^m, \quad g=(x+1)(x^2+1)^2(x^{12}+1)^m ,\]
then $\Gamma  (f,g)$  is an  arithmetic group. \\

\begin{remark} Note that the degree of the representation is $n=6m+5$ 
with $m$ {\it arbitrary}. Therefore, we get infinitely many examples 
of higher rank orthogonal monodromy. \\

The polynomials $f_0,g_0$ of Theorem \ref{maintheorem} may be replaced
by any pair $ f_1,g_1$ for which  $\Q - rank (H)= 2$, where $H=O_5$ is
the Zariski  closure of the  $5\times 5$ hypergeometric  group $\Gamma
(f_1,g_1)$.  Thus the analogue of Theorem \ref{maintheorem} gives many
more examples  of arithmetic  monodromy. In section  \ref{O(2,3)}, 
more examples  of pairs  $f_1,g_1$ are given  for which  the monodromy
group  is arithmetic  in $O(2,3)$  and  the associated  group $G$  has
$\Q$-rank two.
\end{remark}

\begin{remark} Suppose $x=(x_1,\cdots,x_k)\in \Q^k$ and $r\in \Q$. Let
us    write   $x+r$    for    the   $k$-tuple    $(x_1+r,x_2+r,\cdots,
x_k+r)$.  Consider the foregoing  example $f=(x^5-1)(x^{12}+x^6+1)^m$.
Its  parameters are  the $5+12m$  tuple which  is obtained  by pasting
together  the  parameters   $(0,1/5,2/5,3/5,4/5)$  of  $f_0$  and  the
paramteters $(1/3,2/3)/6 +j/6$,  with $j=0,1,2,3,4,5$, with the latter
parameters  repeated  $m$ times  since  $x^{12}+x^6+1$  occurs to  the
exponent $m$. The parameters of $g$ can similarly be worked out.
\end{remark}

\subsection{Description of the Proof} We first show that when $P=Q=1$,
the monodromy  group is  an arithmetic subgroup  of $O(2,3)$.  This is
proved  by  showing that  the  reflection  subgroup  generated by  the
elements  $A^kCA^{-k}$ ($k\in  \Z$), is  arithmetic.   We prove  the
arithmeticity   of   the  reflection   group   $\Delta$  by   explicit
computation, by showing that $\Delta $ contains an arithmetic subgroup
of the  unipotent radical of a parabolic  subgroup.  The arithmeticity
then follows by appealing to a  generalization of a Theorem of Tits on
unipotent    generators   of    arithmetic    groups   (see    Theorem
\ref{bamise}). \\

We then  prove the general case by  using Proposition \ref{bootstrap}.
The proposition  says the  following: if $\Gamma$  is a  Zariski dense
subgroup of  $O(h)(\Z)$ where $h$  is a non-degenerate  quadratic form
over  $\Q$, such  that $\Gamma$  contains a  finite index  subgroup of
$O(W)(\Z)$, for  some $5$ dimensional  $W$ with $\Q-rank  (W)=2$, then
$\Gamma$ is  itself an  arithmetic group. \\

In section \ref{O(2,3)}, we list the pairs $f_1,g_1$ (up to a scalar
shift,  i.e. changing  $f(x)\mapsto  f_1(-x)$ $g_1(x)\mapsto  g_1(-x)$
-see  \cite{Beu-Hec},  and   \cite{FMS})  of  degree  $5$,  satisfying
$f_1(0)=-1,g_1(0)=1$  such  that   the  hypergeometric  group  $\Gamma
(f_1,g_1)$ is  arithmetic. Two  of the groups  $G$ have  $\Q$-rank one
while the  rest have $\Q$-rank  two.  The proof  of arithmeticity in
these cases is similar that of the group  $\Delta $ considered above. \\

The authors of \cite{FMS} also  give a list of hypergeometric $O(2,1)$
and they prove that in each  case the group is arithmetic.  An earlier
version of the present paper  used these computations in \cite{FMS} to
put  together various arithmetic  $O(2,1)$'s to  deduce arithmeticity;
however, this method does not cover  as many cases as the present one,
and in addition, the present proof is more uniform.

\newpage

\section{Preliminary Results}

\subsection{The Quadratic Form $h$}

\begin{notation} We  will view the  quadratic vector space $V$  as the
$\Q$-algebra $V=\Q[x]/(f(x))$  and the operator  $A$ as multiplication
by  $x$; then  with respect  to the  basis $1,x,\cdots,  x^{n-1}$, the
matrix   of   $A$   is   the   companion   matrix   of   $f$.    Write
$V_{\Z}=\Z[x]/(f(x))  \subset V$.  We assume  that $f\in  \Z[x]$  is a
product of cyclotomic polynomials, with $f(0)=-1$. \\

Let  $g\in  \Z[x]$  be   a  product  of  cyclotomic  polynomials  with
$g(0)=1$. Assume $f,g$ have no common root.  We introduce the operator
$B$  on $V$ by  setting $B(w)=A(w)=w$  if $w=1,x,\cdots,  x^{n-2}$ and
$B(x^{n-1})=  x^n-g(x)$; the latter  being a  polynomial of  degree at
most  $n-1$, has been  viewed as  an element  of $V_{\Z}$.   Denote by
$\Gamma =\Gamma (f,g)$ the group  generated by the two matrices $A,B$.
Following \cite{Beu-Hec},  we say that $f,g$ form  an {\it imprimitive
pair}  if  the the  vector  space  $V$ splits  into  a  direct sum  of
subspaces $V_i$ such that each  element of the group $\Gamma$ permutes
these spaces  $V_i$; if not, we  say that $f,g$ form  a {\it primitive
pair}.  We assume henceforth that $f,g$ form a primitive pair.  \\

Let $h$ be  the quadratic form preserved by  $A,B$; by \cite{Beu-Hec},
such a form  exists and is unique up to scalar  multiples. To ease the
notation, we  write $x.y=h(x,y)$  for $x,y\in V$. \\  

The following  observations are taken from  \cite{FMS}, section (2.4).
Since $A$ and  $B$ coincide on the span of  first $n-1$ basis elements
$1,x,\cdots,  x^{n-2}$, it follows  that if  $C=A^{-1}B$, then  $C$ is
identity  on $1,x,\cdots,  x^{n-2}$  and  the image  of  $C-1$ is  one
dimensional. Moreover, $(C-1)(V_{\Z})$ is of  the form $\Z v$ for some
$v\in V_{\Z}$.  Therefore, $(C-1)(x^{n-1})= v$.  Since the determinant
of  $C$  is  $-1$,  we   have  $C^2-1=0$;  it  follows  that  $Cv=-v$.
Consequently,  $v$  is  orthogonal  to $1,x,\cdots,  x^{n-2}$.   Hence
$x^{n-1}.v\neq 0$. We normalise $h$ so that $x^{n-1}.v=1$.  Therefore,
for any vector $w=u_0+u_1x+\cdots+u_{n-1}x^{n-1}\in V$ ($ u_i\in \Q$),
we have $w.v=u_{n-1}$. Denote by  $\lambda $ the linear form $w\mapsto
u_{n-1}$.\\

It  follows that  $Av=  A(C-1)(x^{n-1})=(B-A)(x^{n-1})=g-f$ where  the
latter is viewed as a linear combination of $1,x,\cdots, x^{n-1}$ i.e.
an element of $V_{\Z}$; since $f,g$  are monic of degree $n$, $f-g$ is
a polynomial  of degree not exceeding  $n-1$ and there is  no abuse of
notation.
\end{notation}

\begin{lemma}   \label{innerproduct} {\rm ( see  \cite{FMS},   Proposition
(2.10) )} Under the preceding notation and normalisation of $h$, we have
the formulae
\[v.w=u_{n-1} \quad \forall w\in V,\]
\[v.v=2 \quad {\rm and}\] 
\[C (w)=w-(w.v)v, \quad \forall w\in V.\] 
\end{lemma}

\begin{proof}  The  first  part  is  already proved  (see  the  second
paragraph preceding the  Lemma); so we need only  prove the second and
the third equalities.  Note that (by the first formula) the orthogonal
complement to  $v$ is  exactly the span  of $1,x,\cdots,  x^{n-2}$. We
have   also  seen  that   $Av=g-f=c_{n-1}x^{n-1}+\cdots+c_1x+2$  where
$c_i=b_i-a_i$.  The operator on $V$  given by multiplication by $x$ is
invertible. The equation
\[f(x)=x^n+a_{n-1}x^{n-1}+\cdots+a_1x-1\equiv 0\in V\]
shows that {\it in $V$},  
\[\frac{1}{x}\equiv   x^{n-1}+a_{n-1}x^{n-2}+\cdots  +a_2x+a_1,\]  and
that
\[\frac{1}{x}(x^k)=x^{k-1}  \quad  \forall k  \quad  with \quad  1\leq
k\leq n-1.\] Therefore,
\[v=\frac{1}{x}(Av)= \frac{1}{x}(c_{n-1}x^{n-1}+\cdots +c_1x+2)=\]
\[=c_{n-1}x^{n-2}+\cdots+c_2x+c_1+
2(x^{n-1}+a_{n-1}x^{n-2}+\cdots+a_1).\]  The last equality  shows that
the coefficient of  $x^{n-1}$ in $v$ (viewed as  linear combination of
$1,x,\cdots,x^{n-1}$) is exactly two.  Therefore $v.v=2$.  \\

The operator  $C$ is identity  on $1,x,\cdots,x^{n-2}$ and is  $-1$ on
$v$. The operator $w\mapsto w-(w.v)v$ is also identity on $1,x,\cdots,
x^{n-2}$ since  $v$ is orthogonal to  $1,x,\cdots, x^{n-2}$; moreover,
$v-(v.v)v=-v$; therefore, the third equation in the Lemma follows.
\end{proof}

\begin{remark}  \label{importantremark}  Since  $g$  and  $f$  are  co
-prime,  the element  $Av=g-f$, viewed  as an  element of  the algebra
$V=\Q[x]/((f(x))$, is  invertible. Hence $Av$ is cyclic  for the action
of $A$ and hence so is $v$. Thus, $v,Av, \cdots,A^{n-1}v$ form a basis
of the vector space $V$. In particular, the inner products $A^iv.A^jv$
determine  the quadratic form  $h$; the  invariance of  the quadratic
form  $h$  under $A$  implies  that $h$  is  determined  by the  inner
products $A^iv.v$  ($0\leq i \leq n-1$).   By Lemma \ref{innerproduct}
the latter is just the $x^{n-1}$-th coefficient of $A^iv$, viewed as a
linear combination of $1,x, \cdots, x^{n-1}$.  Hence $h$ is determined
by the ``highest coefficients'' of the remainders of the polynomials 
$g-f,x(g-f), \cdots, x^{n-1}(g-f)$, after division by $f$.
\end{remark}

\begin{lemma} \label{realrank}  Suppose {\rm two} of the  roots of $f$
(or of  $g$) occur with multiplicity  two {\rm (} i.e.  suppose $f$ is
divisible by the square of a  quadratic polynomial {\rm )}. Then $\R -
rank(V)\geq 2$. If $n\geq 7$, then $\Q-rank (V)\geq 2$.
\end{lemma}

\begin{proof} Write 
\[f(x)=\prod _{j=1}^n (x-e^{2\pi i \alpha _j}), \quad g(x)=\prod _{j=1}^n (x-e^{2\pi i \beta _j}).\]
We assume, as we may, that 
\[0\leq \alpha _1\leq \cdots \leq \alpha _n <1 \quad {\rm and} \quad  
0\leq \beta _1\leq \cdots \leq \beta _n <1.\]

If $(p,q)$ is the  signature of the quadratic form, then
the real rank is  $\frac{1}{2}(p+q-\mid p-q\mid)$.  There is a formula
(\cite{Beu-Hec}, or \cite{FMS}) for $\mid p-q\mid$:
\[\mid  p-q\mid  =\mid  \sum  _{j=1}^n (-1)^{j+m_j}\mid  \]  
where $m_j$  is the  {\it number}  of indices $k$  such that  $\beta _k
<\alpha _j$.   \\

If $f$ has one root with multiplicity one i.e. for some index $j$ we have 
$\alpha_j=\alpha _{j+1}$, then $m_j=m_{j+1}$ ($=m$, say). Hence
\[(-1)^{j+m_j}+(-1)^{j+1+m_{j+1}}=(-1)^{j+m}+(-1)^{j+1+m}=0.  \] Hence
two terms in the above expression  for $\mid p-q \mid $ cancel out and
$\mid p-q\mid \leq p+q-2$.\\

Similarly, if  there are {\it  two} roots with multiplicity  two, then
$\mid p-q\mid \leq p+q-4$. Hence 
\[\R-rank (V)=\frac{1}{2}(p+q-\mid p-q\mid)\geq \frac{1}{2}(4)=2.\] 

The  second  part  of  the  Lemma,  is  an  easy  consequence  of  the
Hasse-Minkowski theorem: if a rational quadratic form in at least five
variables represents a real zero, then it represents a rational zero.
\end{proof}

\subsection{The Quadratic Forms $h$ and $h_0$}

\begin{notation}     We     now     consider     $f_0=(x^5-1)$     and
$g_0=(x+1)(x^2+1)^2$;  they form a  primitive pair  by \cite{Beu-Hec}.
Denote the  associated monodromy  group by $\Delta  =\Gamma (f_0,g_0)$
generated   by   the  companion   matrices   $A_0,B_0$  of   $f_0,g_0$
respectively;  it  is  a   subgroup  of  $GL_5(\Z)$  and  preserves  a
non-degenerate quadratic form  $h_0$ on $V_0=\Q ^5=\Q[x]/(f_0(x))$. We
have the vector $v_0\in V_0$ as before (we have denoted the vector $v$
of the  previous subsection  in this case  ($n=5$) by $v_0$).  We view
elements   $w$  of   $V_0$  as   polynomials  of   degree   $\leq  4$:
$w=u_4x^4+u_3x^3+u_2x^2+u_1x+u_0$ with $u_i\in \Q$. Denote by $\lambda
_0$ the linear  form $w\mapsto u_4$.  It is easily  seen that $V_0$ is
spanned by the vectors $v_0,A_0v_0,A_0 ^2v_0,A_0 ^3v_0,A_0 ^4v_0$. \\

Fix an integer  $m\geq 0$ and put $n=6m+5$. Let  $P,Q\in \Z[x]$ be two
(monic) polynomials  of degree $m$,  which are products  of cyclotomic
polynomials  such   that  $P(0)=Q(0)=1$.   Consider   the  polynomials
$f(x)=f_0(x)P(x^6)$ and  $g(x)=g_0(x)Q(X^6)$.  Then $f,g$  have degree
$n=6m+5$.  We assume  that $f,g$ are co-prime and  that $(f,g)$ form a
primitive  pair.  Let  $\Gamma  =\Gamma (f,g)$  be the  hypergeometric
group.  We have the vector  $v$ in $V=\Q[x]/(f(x))$. Denote by $W$ the
span of the vectors  $v,Av,A^2,A^3v,A^4v$. We have an injective linear
map $i: V_0\ra V$  given on the basis elements  
$\{ A_0^kv_0 ; 0\leq k \leq 4\}$ by the formula $A_0^kv_0\mapsto A^kv$.
\end{notation}

We first prove an easy preliminary Lemma.

\begin{lemma} \label{prelim} If $k\leq 4$, then the $x^{n-1}$-th 
coefficient of the remainder $R(x)$ of $x^k(f-g)$ upon division by $f$
is  the same  as $x^4$-th  coefficient  of the  remainder $R_0(x)$  of
$x^k(f_0-g_0)$ upon division by $f_0$.
\end{lemma}

\begin{proof} 
We  may write $R_0(x)=x^k(f_0-g_0)+q_0(x)f_0(x)$  for some  $q_0$, and
$R(x)=x^k(f-g)+q(x)f(x)$ for some $q$. Consider the equation
\[x^k(f-g)=x^k(f_0P(x^6)-g_0Q(x^6))= \] 
\[=x^k(f_0-g_0)P(x^6)+x^kg_0(P(x^6)-Q(x^6)).\] 
Since  $P,Q$  are   monic  of  degree  $m$  in   $x$,  the  degree  of
$P(x^6)-Q(x^6)$  is   at  most   $6m-6$.  Therefore,  the   degree  of
$x^kg_0(P(x^6)-Q(x^6))$ is  at most $6m-2+4$ and hence  does not contribute
to the $x^{n-1}=x^{6m+4}$-th term. \\

The     polynomial     $x^k(f_0-g_0)P(x^6)$     may     be     written
$R_0(x)P(x^6)+q_0(x)f_0(x)P(x^6)$; since $f=f_0P(x^6)$ it follows that
the remainder  $R=R_0P(x^6)+x^kg_0(P(x^6)-Q(x^6))$. Hence, by the last conclusion 
of the preceding paragraph,   the coefficient of  $x^{4+6m}$ in
$R(x)$ is the coefficient of $x^4$ in $R_0(x)$.  The Lemma follows.

\end{proof}

\begin{lemma} \label{isometry} The above linear map $i$ is an isometry
of  the  quadratic  spaces  $(V_0,h_0)$  and  $(W,h_{\mid  _W})$.   In
particular,the   restriction  $h_{\mid   _W}$   of  $h$   to  $W$   is
non-degenerate;   write   the   orthogonal  decomposition   $V=W\oplus
W^{\perp}$.     The    group     generated    by    the    reflections
$A_0^rC_0A_0^{-r};0\leq r \leq 4$ is isomorphic to the group $\Delta $
generated by the reflections  $A^rCA^{-r}: 0\leq r \leq 4$.  Moreover,
$\Delta $ acts trivially on $W^{\perp}$.
\end{lemma}

\begin{proof} Since the map $i$  is linear, to check isometry, we need
only    check   that   the    inner   products    $h(A^kv,A^lv)$   and
$h_0(A_0^kv_0.A_0^lv_0)$  coincide if  $0\leq k,l  \leq 4$.  Using the
invariance of $h,h_0$ under $A,A_0$ it is sufficient to check that the
inner  products  $h(v,A^kv)$ and  $h_0(v_0,A_0  ^kv_0)$ coincide.  But
these values  are nothing but the $x^{n-1}$-coefficient  of the vector
$A^kAv=x^k(f-g)$,  viewed  as  a  linear combination  of  $1,x,\cdots,
x^{n-1}$ (similarly $h_0(A_0^kv_0,v_0)$ is the $x^4$-th coefficient of
the vector  $A_0^kA_0v_0=x^k(f_0-g_0)$ viewed as  a linear combination
of  $1,x,x^2,x^3,x^4$).  \\  

By Lemma \ref{prelim},  $x^{n-1}$ coefficient of $x^k(f-g)$ is the
same as  $x^4$-th coefficient of  $x^k(f_0-g_0)$. Therefore the
Lemma follows.
\end{proof}

\section{A Bootstrapping Step for integral Orthogonal Groups}

In  this section,  we prove  a  result (Proposition \ref{bootstrap})  which  will 
be  used in  the proof of Theorem \ref{maintheorem}.  The  result says that 
a subgroup of the
integral orthogonal  group has finite  index if it contains  finite index
subgroups of smaller integral orthogonal  groups.  \\

Let   $h$  be   a  non-degenerate   rational  quadratic   form   on  an
$n$-dimensional $\Q$ vector space $V$.  Suppose that $W\subset V$ is a
$5$  dimensional   subspace  on  which  the  restriction   of  $h$  is
non-degenerate.   and  such  that   if  $V=W\oplus  W^{\perp}$  is  an
orthogonal decomposition, then $O(W)$ may be viewed as the subgroup of
$V$ which  fixes $W^{\perp}$ point-wise. Assume  that $V_{\Z}\subset V$
is a lattice  on which $h$ takes integral  values. Denote by $O(V,\Z)$
the integer points of $O(V)$; define $O(W,\Z)$ similarly.

\begin{proposition} \label{bootstrap}  If $\Gamma  $ is a  Zariski dense
subgroup of  $SO(V,\Z)$ whose  intersection with $SO(W,\Z)$  has finite
index  in $SO(W,\Z)$,  then $\Gamma$  has finite  index  in $SO(V,\Z)$
provided $\Q-rank (W)= 2$
\end{proposition}

\subsection{Unipotent generators for arithmetic groups}

The following  theorem (see \cite{Ra},  \cite{Ve}) is an  extension to
all simple groups,  and all opposing parabolic subgroups,  of a result
of Tits (the result of  Tits \cite{Ti} was proved for Chevalley Groups
of $K$ rank at least two).

\begin{theorem}  \label{bamise} Suppose  $G$ is  an  absolutely almost
simple linear  algebraic group defined  over a number field  $K$, such
that $K$-rank  of $G$ is $\geq  1$ and $G(O_K)$ has  higher real rank,
i.e.
\[\infty  - rank  (G)  \stackrel  {def}{=} \sum  _{v \mid  \infty}
K_v-rank (G) \geq 2.\] Suppose  $P$ is a parabolic $K$-subgroup of $G$
with  unipotent   radical   $U$  and   let  $P^{-}$  be   a  parabolic
$K$-subgroup  defined  over $K$  and  opposed  to  $P$ with  unipotent
radical $U^{-}$.  Let  $\Gamma \subset G(O_K)$  be a  subgroup which
intersects $U(O_K)$ in  a finite index subgroup {\rm  (} and similarly
with $U^{-}(O_K)${\rm )}. Then $\Gamma $ has finite index in $G(O_K)$.
\end{theorem}

\subsection{Algebraic Groups} \label{alggroups}

The reference for the material in this subsection is \cite{Bor-Ti}. 

\begin{notation}  Let  $G$ be  a  $\Q$-simple  linear algebraic  group
defined and isotropic  over $\Q$. Fix a maximal split  torus $S$ and a
minimal  parabolic subgroup  $P_0$ containing  $S$. Under  the adjoint
action of  $S$, the Lie  algebras $\fp _0$  and $\fg$ of $P$  and $G$,
decompose   as  follows  ($\Phi   ^+$  is   the  system   of  positive
roots,i.e.those {\it  roots} occurring in  $\fp _0$, and $\Phi$  is the
system of all roots)
\[\fp _0 = \fg _0\oplus _{\alpha \in \Phi ^+} \fg _{\alpha}, \quad \fg
_0=  \fg _0  \oplus _{\alpha  \in \Phi  ^+} \fg  _{\alpha  }\oplus \fg
_{-\alpha}.\]  Given  a  root  $\alpha$  denote  by  $U_{\alpha}$  the
unipotent algebraic  subgroup generated by  elements of the  form $exp
(X)$ for all $m\geq 1$ and all $X\in \fg _{m \alpha}$. Given $X\in \fg
_{\alpha}(\Q)$  we  get  a  one parameter  unipotent  group  $t\mapsto
X_{\alpha}(t):{\mathbb G}_a\ra U_{\alpha}$. 
\end{notation}

Suppose that  $\alpha, \beta  $ are two  roots which are  not rational
multiplies of each other. Then  necessarily $\Q -rank (G)\geq 2$. In a
group denote by $[x,y]$ the commutator element $xyx^{-1}y^{-1}$. \\

Before stating  the next  Lemma, we  fix some notation.  Let $X$  be a
finite totally  ordered indexing set and  $G$ a group.  Let $g_x\in G$
for each index $x\in X$. We let
\[P=\prod  _{x\in X}  g_x,\] denote  the  product of  $g_x$ where  the
product is taken  in the sense that if $x<x'\in  X$ then $g_x$ appears
to the left of $g_{x'}$.

We have the Chevalley commutator relations:  

\begin{lemma}  \label{chevalleycommutator}  Suppose   that  $G$  is  a
Chevalley  group. For  every pair  of integers  $m,n\geq 1$  such that
$m\alpha+n\beta  $ is  a  root, there  exist  one parameter  unipotent
groups $X_{m\alpha +n \beta}$ such that for all $s,t\in {\mathbb G}_a$
we have the commutator relation
\[  [X_{\alpha }(s),  X_{\beta }  (t)]=  \prod X_{m\alpha  +n \beta  }
(s^mt^n),\] where the product is of elements of the group $G$, and the
roots  $\theta  _{m,n}=  m\alpha  +n  \beta $  are  arranged  in  some
arbitrary but  fixed order (so  that if $m+n<m'+n'$ then  $\theta _{m,n}$
appears  before,  i.e.  to  the  left of,  $\theta  _{m',n'}$  in  the
product).
\end{lemma}

\begin{lemma}  \label{chevalleyintegral}   Let  $G$  be   a  Chevalley
group. Fix an integer $N\geq 1$ and consider the group $U_N$ generated
by commutators
\[[X_{\alpha }(s),X_{\beta }(t)];s,t \equiv 0 \quad (mod \quad N).\] in
$G(\Z)$. Then  there exists  an integer $M\geq  1$ such that  for {\bf each}
$m,n\geq  1$ the  group  $U_N$ contains  the  subgroup $X_{m\alpha  +n
\beta}(x)$ for $x\equiv 0 \quad (mod \quad M)$.

\end{lemma}

\begin{proof} The Zariski closure $U_0$ of the group $U_N$ of integral
matrices is  a unipotent group:  in view of the  commutator relations,
the root groups $U_{m\alpha  +n\beta}:m,n \geq 1$ generate a unipotent
group, say $U^*$. Now a Zariski dense subgroup of $U_0(\Z)$ has finite
index in $U_0(\Z)$  (see \cite {Ra1}).  Hence we  need only prove that
the one parameter groups $X_{m,n}$ lie in $U_0$ ($X_{m,n}$ denotes the
group $X_{m\alpha +n\beta}$). Denote by  $\fu _0$ and $\fu ^*$ the Lie
algebras of  $U_0$ and of  $U^*$. Denote by  $log$ the inverse  of the
exponential map from the unipotent group $U^*$ onto $\fu^*$.  \\\

Taking the  logs of  the commutators in  the commutator  relations, we
obtain that a polynomial $P$ in $t,s$ namely
\[P(t,s)=log  (\prod _{m,n}  (X_{m,n}(t^ms^n))\] takes  values  in the
subspace  $\fu _0$;  hence the  coefficient of  $ts$ also  does.  This
coefficient is  nothing but $log  (X_{1,1})$; hence the  first element
$X_{1,1}$ in the ordering lies in  $U_0$; now an easy induction in the
ordering  implies that all  the $X_{m,n}$  lie in  $U_0$ (to  ease the
notation, we have denoted by  $X_{m,n}$ the image of the $1$-parameter
group $X_{m,n}$ in $U^*$).
\end{proof}

\subsection{The Special Case of $O(2,3)$}  We now assume that $h_0$ is
a non-degenerate quadratic form on a $5$-dimensional $\Q$ vector space
$W$ with $\Q -rank (W)=2$.   Denote by $w.w'$ the element $h_0(w,w')$.
The   rank   assumption   implies    that   There   exists   a   basis
$\e_1,\e_2,w_3,\e_2^*,\e_1^*$  of $W$ such  that $\e_1  ^2=\e_2 ^2=0$,
$(\e_2 ^*)  ^2=(\e_1^*) ^2=0$, $w_3^3\neq  0$, $\e_i.w_3=\e_i^*.w_3=0$
for  $i=1,2$ and  $\e_i(\e_j^*)=\delta  _{ij}$. With  respect to  this
basis, $O(W)$ may  be thought of as a subgroup of  the group $GL_5$ of
$5\times 5$ matrices. We will informally denote $O(W)$ by $O(2,3)$. \\

The intersection of the diagonals  with $H= O(W)$ is a two dimensional
split  group $S=\{t_1,t_2)\in {\mathbb  G}_m ^2\}$  which acts  by the
characters $t_i$ on $\e_i$ , $t_i^{-1}$ on $\e _i ^*$ and trivially on
$w_3$. Denote by  $\fh _{\pm x_i}$ (resp. $\fh  _{\pm (x_1+x_2)}$) the
subspace of the Lie algebra $\fh $ of $O(W)$, on which $S$ acts by the
character $t_i^{\pm 1}$ (resp. $(t_1t_2)^{\pm 1}$).

As a special case of Lemma \ref{chevalleyintegral}, we have 

\begin{lemma} \label{O_5specialcase} In  the notation of the preceding
subsection,  for  any integer  $N\geq  1$,  the  group generated  by  the
commutators
\[ [X_{-x_1}(s),X_{x_1+x_2}(t)]  ;  s,t   \equiv  0\quad(mod  \quad  N)  \]
contains the group $X_{x_2}(M\Z)$ where  the latter is the subgroup of
$X_{x_2}(\Z)$ of elements congruent to the identity modulo $m$.
\end{lemma}
\begin{proof}  We  need only  note  that $x_2=1(x_1+x_2)+1(-x_1)$  and
apply  Lemma \ref{chevalleyintegral}  to the  roots $\alpha=x_1+x_2$
and $\beta =x_1$.
\end{proof}

\subsection{The group $O(V)$} Since $\Q$  rank of $W$ is two, the $\Q$
rank $r$  of $V$ is at  least two. There exists  a basis $\e_1,\cdots,
\e_r,v_1,  \cdots, v_m,\e_r  ^*,  \e_{r-1}^*, \cdots,  \e_1^*$ of  $V$
($Dim  (V)=n=2r+m$)  such  that  (1) $\e_j.\e_j^*=\delta  _{ij}$  (the
Kronecker   delta    symbol)   and   (2)   $W$   is    the   span   of
$\e_1,\e_2,\e_1^*,\e_2^*$  and  an element  $w_3$  which  is a  linear
combination of the vectors $v_j$  and the vectors $\e_3, \cdots, \e_r,
\e_r^*,  \cdots,  \e_3 ^*$.   Under  the  inclusion $O(W)\subset  O(V)
\subset GL_n$, the  torus $S$ is the subgroup  of diagonal matrices in
$GL_n$  which  act by  the  characters  $t_1,t_2$  on $\e_1,\e_2$,  by
$t_1^{-1} t_2^{-1}$  on $\e_1^*, \e_2^*$ and by  the trivial character
on all the other basis elements above. \\

Let  $X$ be the  span of  $\e_1,\e_2$ and  $Y$ the  span of  the basis
vectors  $\e_3,\cdots, \e_r,v_1,\cdots,  v_m,  \e_r^*,\cdots, \e_3^*$.
Let $M$ be the subgroup of  $O(V)$ which stabilises the spaces $X$ and
$Y$;  let  $P$ be  the  subgroup  which  stabilises the  partial  flag
$X\subset X^{\perp}= X\oplus Y \subset  V$ and $U$ the subgroup of $P$
which acts trivially  on successive quotients of this  flag. Since $X$
is totally isotropic,  $P$ is a parabolic subgroup,  $U$ its unipotent
radical. $M$ is a Levi subgroup of $P$ containing the torus $S$ and we
have  $P=MU$. It  is easy  to see  that $M=GL_2.O(Y)$  where $S\subset
GL_2$  and $GL_2$  is  the subgroup  of  $M$ which  acts trivially  on
$Y$. \\

With respect  to the  adjoint action  of $S$, the  Lie algebra  of $U$
splits into  the character spaces  $\fg_{x_1}$, $\fg _{x_2}$  and $\fg
_{x_1+x_2}$;  it  is  easy  to   see  that  $\fg  _{x_1+x_2}$  is  one
dimensional  and  is  $\fh   _{x_1+x_2}$.   Moreover,  the  action  of
$M=GL_2.O(Y)$ on the direct sum $\fg _{x_1}\oplus \fg _{x_2}$ is simply
the   exterior  tensor   product  $St\otimes   St$  of   the  standard
representations  of   $GL_2$  and   $O(Y)$  (and  in   particular,  is
irreducible  for the  action  of $M$).   We  note that  $\fh \cap  \fg
_{x_1}\neq \{0\}$.

\begin{lemma} \label{abelian} For  any $u\in U$ and any  $v \in U$, we
have  $uvu^{-1}=v.u'$ where  $u'\in X_{x_1+x_2}$.  Moreover,  given an
integer  $N$,   there  exists   a  power  $v^M$   of  $v$   such  that
$uv^Mu^{-1}v^{-M}$ lies in $X_{x_1+x_2}(N\Z)$.
\end{lemma}

\begin{proof} The  first part is just  a restatement of  the fact that
$\fu /\fg_{x_1+x_2}$ is abelian. The second part is an easy consequence.  
\end{proof}

\subsection{Proof of the Proposition} Let $\mathcal U$ denote the open
Bruhat   cell  $Pw_0U$  where   $w_0$  is   the  longest   Weyl  group
element. Since $\Gamma$ is Zariski  dense in $G=O(V)$, it follows that
$\Gamma \cap {\mathcal U}$ is also Zariski dense in $G$. Given $\gamma
\in  \Gamma \cap  {\mathcal U}  \subset Pw_0U$  write  $\gamma =pw_0u$
accordingly. Then the elements $p$  (as $\Gamma $ varies) consist of a
Zariski dense  subset of $P$. We fix  a finite set $F$  of elements of
$\Gamma \cap {\mathcal U}$ such that the span of the conjugates $p(\fh
)_{x_2}p^{-1}$ contains  all of $\fg  _{x_1}\oplus \fg _{x_2}$;  it is
possible to find such a finite  set since $P$ acts irreducibly on $\fu
/\fg  _{x_1+x_2}$  and  the  $P$  parts $p$  of  elements  of  $\Gamma
\cap{\mathcal  U}$  are  Zariski  dense  in  $P$ (since  $\Gamma  \cap
{\mathcal U}$ is Zariski dense in $G$). \\

Since $\Gamma$ contains a finite index subgroup of $H(\Z)$ it contains
the  congruence  group $(U\cap  H)(N\Z)$  for  some  integer $N$.   In
particular,  there exists an  element $v\in  (U\cap H_{x_1})(M\Z)$  
with the integer $M$ large  such that for all $\gamma \in F$  the finite set of
the  previous   paragraph,  the  elements   $uvu^{-1}v^{-1}$  lie  in
$U_{x_1+x_2}(N\Z)\subset \Gamma$. Consider the commutator set
\[E= [^{\gamma}(vX_{x_1+x_2}(M\Z)), X_{x_1+x_2}(M\Z)].\] Since $\gamma
=pw_0u$ it follows from Lemma \ref{abelian} that this set contains the
commutator set \[ [^{pw_0}(v),X_{x_1+x_2}(M\Z)].\]  We note that for a
large enough $M$,  the conjugate $^{p^{-1}}(X_{x_1+x_2}(M\Z))$ lies in
$X_{x_1+x_2}(N\Z)$.  Therefore, $E$ contains the commutator set
\[  ^{p}([^{w_0}(v), X_{x_1+x_2}(N\Z)]).\]  
Since $v\in X_{x_1}$, it follows that $^{w_0}(v)$  lies in
$X_{-x_1}(N\Z)$. It follows from Lemma  \ref{O_5specialcase} that the
group generated by the  latter commutators contains $X_{x_2}(M\Z)$ for
some integer $M$ divisible large powers of $N$ and the denominators of
the rational matrix $p$ .   Since the $p$ conjugates of $X_{x_2}(M\Z)$
generate (modulo  centre) all of  $\fg_{x_1}\oplus \fg _{x_2}$  in the
lie  algebra  of  the  Zariski  closure, it  follows  that  the  group
generated  by  the  $p$  conjugates  of  our  commutator  set  contain
$U_{x_1}(M\Z)$ and $U_{x_2}(M\Z)$, where  $p=p(\gamma) $ and $\gamma $
runs through a  (possibly large) finite set in  $\Gamma \cap {\mathcal
U}$.  These  generate a finite index subgroup  of $U(\Z)$.  Therefore,
$\Gamma $ contains a finite index subgroup of $U(\Z)$. \\

Now $U$ is the unipotent radical of the parabolic subgroup $P$ and $G$
has  real  rank  (even  rational  rank)  at  least  two.   By  Theorem
\ref{bamise}, it follows that $\Gamma$ is arithmetic.

\section{The arithmeticity of $\Gamma$}

\subsection{Arithmeticity of $\Delta$} \label{deltaarithmetic}

In  this subsection,  we prove that  the group  $\Delta =\Gamma
(f_0,g_0)$ is  an arithmetic subgroup of $O(W,h_0)$  and that $\Q-rank
(W)=2$. To ease the notation,  we drop the subscript $_0$ in $v_0,A_0$
and simply write $v,A$ etc.
\begin{lemma} \label{O_5lemma} We have
\[v.v=2,Av.v=1,A^2v.v=2,A^3v.v=2, A^4v.v=1.\]  The vector $\e= v-A^2v$
is isotropic and  the orthogonal complement $\e ^{\perp}$  is the span
of  the  four vectors  $\e,  v,  Av,  v'=A^3v+A^4v-v$.  Moreover,  the
reflections  about  the  three  vectors $v,Av,v'=A^3v+A^4v-v$  lie  in
$\Gamma$ and fix $\e$.
\end{lemma}
\begin{proof} We view $V$ as  the space of polynomials of degree $\leq
4$.   We need only  compute the  coefficient of  $x^4$ of  the vectors
$v,Av,A^2v,A^3v,A^4v$. Since  $Av=g-f$, and $g=(x+1)(x^2+1)^2,f=x^5-1$
we have  $Av= x^4+2x^3+2x^2+x+2$. Hence the $x^4$  coefficient of $Av$
is  $1$.   Moreover,  $A^2v=x^5+2x^4+2x^3+x^2+2x=  2x^4+2x^3+x^2+2x+1$
(since $x^5-1\equiv  0$ in $V$).  Therefore, the $x^4$  coefficient of
$A^2v$ is $2$. The others are proved similarly. \\

The second part follows immediately from the first: as an illustration, 
we compute 
\[\e.v=A^2v.v-v.v=2-2=0 \quad {\rm and}\] 
\[\e.A^2v=A^2v.A^2v-v.A^2v=v.v-A^2v.v=2-2=0. \] Note  that we have used
the invariance of  the ``dot product'' under the  action of $A,B$. The
others are proved similarly. \\

We denote by  $\Delta (\e)$ the subgroup of $\Delta  $ which fixes the
line  through $\e$.   It is  the intersection  of $\Delta  $  with the
parablic subgroup  $P=P(\e)$ of $G$  which fixes the  isotropic vector
$\e$. Since the vectors $v,Av,v'$  are in the orthogonal complement of
$\e$  it follows  that the  reflections $C_v,C_{Av},C_{v'}$  about the
vectors  $v,Av,v'$ fix  the  vector  $\e$ and  in  particular, lie  in
$P$. Since the reflections about $v$ and $Av$ are the elements $C$ and
$ACA^{-1}$,  it follows  that $C=  C_v$ and  $C_{Av}=ACA^{-1}$  lie in
$\Delta (\e)$,  since $\Delta $ is  the group generated  by $A,C$.  We
need only prove that $C_{v'}$ lies in $\Delta$. \\

In  a  group,  denote  $^x(y)=xyx^{-1}$.   It is  easy  to  show  that
$^{C_w}(C_{w'})=C_{w'-(w'.w)w}$ for vectors  $w\in V$ with $w.w=2$. We
use  this observation, and  the formulae  for the  dot product  in the
preceding Lemma to compute
\[^{C_{A^3v}C_v}(C_{A^4v})=                         ^{C_{A^3v}}(C_{A^4v-v})=
C_{A^4v-A^3v-v+2A^3v}=C_{v'}.\]  The extreme left hand
side term of  the above equalities lies in $\Delta$  since each of the
reflections  about  $v,A^3v,A^4v$  does.  Therefore,  $C_{v'}$  lies  in
$\Delta$.
\end{proof}

\begin{notation}  The  element $\e$  of  the  Lemma \ref{O_5lemma}  is
isotropic.  Hence we have the  partial flag $\Q \e \subset \e ^{\perp}
\subset  V$. The subgroup  of $O(V)$  which preserves  this flag  is a
parabolic subgroup $P$ and the  subgroup which preserves this flag and
acts  trivially  on  successive  quotients is  its  unipotent  radical
$U$.  Consider the  elements $C_{A^2v}$  and $C_v$.  They fix  $\e$ by
Lemma \ref{O_5lemma}. If $w\in V$, then $C_{A^2v}(w)=w-(w.A^2v)Av$ and
$C_v(w)=w-(w.v)v$  since  $A^2v.A^2v=  v.v  =2$.  Moreover,  if  $w\in
\e^{\perp}, w.(A^2v-v)=0$ i.e.
\[C_{A^2v}C_v(w)=C_{A^2v}(w-(w.v)v)=\]
\[=w-(w.A^2v)A^2v-(w.v)v+(w.v)(A^2v.v)A^2v=\]
\[=w-(w.v)A^2v-(w.v)v+2(w.v)A^2v=w+(w.v)(A^2v-v),\]  since  $A^2v.v=2$
(Lemma \ref{O_5lemma}).  That is
\[C_{A^2v}C_v(w)=w+(w.v)(A^2v-v) =w+(w.v)\e  \quad \forall  w\in \e  ^{\perp}.\] In
particular,    $u=C_{A^2v}C_v$   is    a   non-trivial    element   of
$O(\e^{\perp})$ since its value on $v$ is $v+2(A^2v-v)=v+2\e \neq v$. Moreover,
the  formula  for  $u=C_{A^2v}C_v$  shows  that on  the  quotient  $\e
^{\perp}/\Q  \e$, the  action  of $u$  is  trivial; thus  we have  one
non-trivial unipotent element $u$ in $\Delta \cap U$. 
\end{notation}

\begin{lemma}  \label{O_5unipotent}  The  intersection  of  the  group
$\Delta  (\e)$ with $U$ has finite index in $U(\Z)$.
\end{lemma}

\begin{proof}  The quotient of  the group  $P=P(\e)$ by  its unipotent
radical is  the orthogonal group  $M=O(\e ^{\perp}/\Q \e)$.   By Lemma
\ref{O_5lemma},  the  image  of  $\Delta  (\e)$ in  the  quotient  $M$
contains   the   reflections  about   the   basis  elements   $v,A^2v,
v'=A^3v+A^4v-v$.    Moreover,  these   basis  elements   are  mutually
non-orthogonal.  Hence the group  generated by these three reflections
acts  irreducibly  on  the  standard  representation  $\Q^3$  of  $M$.
However, the conjugation action of  $M$ on $U$ is clearly the standard
representation  of  $M$. Moreover,  by  the  paragraph preceding  this
lemma, $U\cap \Delta (\e)$ contains the non-trivial element $u$; hence
the  conjugates of  $u$  by these  three  reflections spans  $U(\Q)=\Q
^3$. That is, $U\cap \Delta (\e)$  contains a spanning set of $\Q ^3$.
Hence $U\cap \Delta (\e)$ contains a finite index subgroup of $\Z ^3$.
\end{proof}

\begin{proposition}   \label{O_5proposition}    If   $f_0=x^5-1$   and
$g_0=(x+1)(x^2+1)^2$, then the  hypergeometric monodromy group $\Delta
=\Gamma (f_0,g_0)$ is an arithmetic subgroup of $(V_0,h_0)$. Moreover,
$\Q -rank(V_0)=2$.
\end{proposition}

\begin{proof}  Consider  the vector  $\e  '=  A^3v+A^4v-Av$. By  Lemma
\ref{O_5lemma},  $\e '$  lies in  $\e  ^{\perp}$. We  compute the  dot
product of $\e '$ with itself (we write $w^2$ for $w.w$):
\[\e  '.\e   '=(A^2v+A^3v-v)^2=  2+2+2+2A^2v.A^3v-  2A^2v.v-2A^3v.v.\]
Using the formulae  for the dot product in  Lemma \ref{O_5lemma}, this
is  $6+2v.Av-2.2-2.2=6+2-4-4=0$.   Hence  $\e,  \e   '$  are  linearly
independent  mutually orthogonal  isotropic vectors.  Hence  $\Q -rank
(V_0)\geq 2$. Since $h=h_0$ is non-degenerate (\cite {Beu-Hec}) and is
a quadratic form  in $5$ variables, it follows that  $\Q -rank (V)$ is
exactly $2$.\\

By  Lemma \ref{O_5unipotent},  the  intersection $\Delta  \cap U$  has
finite  index  in  $U(\Z)$ and  $U$  is  the  unipotent radical  of  a
parabolic  $\Q$-subgroup  of  $H=O(V_0,h_0)$. By  \cite{Beu-Hec},  the
monodromy group  $\Delta $ is Zariski  dense in $H$.  By the preceding
paragraph,  $\Q$ rank  of $H$  is at  least two.Therefore,  by Theorem
\ref{bamise}, $\Delta $ is an arithmetic subgroup of $H(\Z)$.
\end{proof}

\subsection{Proof of the Main Theorem}  We now return to the situation
of  Lemma  \ref{isometry}. We  have  the  spaces  $V_0$ which  is  $5$
dimensional and  $V$ which  has dimension $n=6m+5$.   We also  have an
isometry $i: V_0\ra  W\subset V$ where $W$ is the image  of $i$ and is
the span  of $v,Av,A^2v,A^3v, A^4v$.  The space  $W$ is non-degenerate
since $V_0$ is. Hence  we have the orthogonal decomposition $V=W\oplus
W^{\perp}$. We may view $O(W)$ as a subgroup of $O(V)$ which leave $W$
stable and act trivially on $W^{\perp}$.  The reflections with respect
to the vectors $\{A^kv: 0\leq k \leq 4\}$ lie in $O(W)$ since they act
trivially on  $W^{\perp}$. By Proposition  \ref{O_5proposition}, these
reflections generate an arithmetic subgroup of $O(W)$. \\

Now  $\Gamma $  is  Zariski  dense in  $O(V)$  by \cite{Beu-Hec}.   By
Proposition \ref{bootstrap}, and by the last sentence of the preceding
paragraph, $\Gamma$ is arithmetic.

\section{Examples of Arithmetic monodromy in $O(2,3)$}} \label{O(2,3)}

In this section,  we list some more cases  of hypergeometric monodromy
when the underlying vector space is $5$ dimensional, the group $G$ has
$\R$-rank  two, and  the  monodromy group  $\Gamma  =\Gamma (f,g)$  is
arithmetic.  In  some  examples,  the  $\Q$-rank of  $G$  is  one  and
sometimes it is $2$. \\

\subsection{Generalities} Suppose  $f,g\in \Z[x]$ are  monic of degree
$5$ and  are a primitive  hypergeometric pair.  Assume  that $f(0)=-1$
and $g(0)=1$.  Put  $V=\Q[x]/(f(x))$; let $v,Av$ be as  before and and
normalise the inner product $h$ so that $v.v=2$.

\begin{lemma}  \label{technical} Suppose  there  exists $g\in  \Gamma$
such that $\e =g(v)\pm v \in V$ is isotropic, and orthogonal to $g(v)$
and $v$. Suppose  that the quotient $\e ^{\perp}/\Q  \e$ is spanned by
three   vectors   of   the   form  $v,g_2(v),g_3(v)$,   with   $g_i\in
\Gamma$. Suppose $v.g_i(v)\neq 0, \quad (i=2,3)$ and $g_2(v).g_3(v)\neq 0$. 
Suppose $\R-rank (V)=2$. Then $\Gamma $ is arithmetic.
\end{lemma}

\begin{proof} We  only work out  the case $\e=g(v)-v$, the  other case
being similar. The  isotropy of $\e$ means that  $g(v).v=2$. Let $w\in
\e ^{\perp}$.  Then $C_{g(v)}(w)= w-(w.g(v)g(v)$.  Since $(w,g(v)-v)=0$
it  follows that  $C_{g(v)}(w)=w-(w.v)g(v)$.   Also $C_v(w)=w-(w.v)v$.
Therefore $C_{g(v)}(w)-C_v(w)=(w.v)(g(v)-v)=(w.v)\e$. Hence $C_{g(v)}$
and $C_v$ coincide  on the quotient $Q=\e^{\perp}/\Q \e$.  but are not
equal: the  value of their difference  on $v$ is  just $2\e$.  Hence
$\theta=C_v^{-1}C_{g(v)}$  lies in  the unipotent  radical $U$  of the
parabolic subgroup $P$ which fixes $\e$. \\

The Levi  part of this parabolic  is $O(Q)$ where $Q$  is the quotient
$\e^{\perp}/\Q \e$.  The group  $R$ generated by the reflections $C_v,
C_{g_2(v)}$  and  $C_{g_3(v)}$   acts  irreducibly  on  $\Q^3$,  since
$g_i(v)$  and  $v$  generate  $Q$  and  are  assumed  to  be  mutually
non-orthogonal.  The action  of $O(Q)$  on  $U$ is  just the  standard
representation,  and  $U\cap  \Gamma$  is not  identity:  it  contains
$\theta$ .  The irreducibility of the action of $R$ on $\Q ^3$ implies
that the conjugates  by elements of $R$ of  $\theta$ generate a finite
index subgroup of $U(\Z)$. \\

Now,  $\R -rank  (G)=2$  and $\Gamma  $  is Zariski  dense  in $G$  by
\cite{Beu-Hec}. By the conclusion  of the preceding paragraph, $\Gamma
$ contains  a finite index  subgroup of $U(\Z)$. The  arithmeticity of
$\Gamma$ follows from Theorem \ref{bamise}.
\end{proof}

\begin{remark} The real rank of $G$ was $2$. The $\Q -rank (G)$ may be
either $1$ or $2$ (by the Hasse-Minkowski theorem). The proof does not
distinguish between these cases.
\end{remark}

\begin{example} $f=(x-1)(x^2+1)^2$ and $g=(x+1)(x^2-x+1)^2$. \\ 

We will show that the  hypergeometric group $\Gamma$ is arithmetic and
that $\Q$  rank of $V$  is two.  As before, we  view $V=\Q ^5$  as the
algebra   $\Q  [x]/(f(x))$,  and   $A$  as   the  operator   which  is
multiplication by $x$. Thus $f\equiv 0$ in $V$, and hence we have
\[f(x)= x^5-x^4-2x^3+2x^2-x+1\equiv 0.\] In other words, 
\[x^5=x^4-2x^3+2x^2-x+1.\]  We   will  view   $V$  is  the   space  of
polynomials of degree $\leq 4$. We have 
\[Av=g-f= -x^3+3x^2-2x+2\]  and, after a  normalisation, the quadratic
form  $h$   is  such  that  $v.v=2$.   Moreover,   after  fixing  this
normalisation  for $h$,  for  any  $w\in V$,  $w.v$  is precisely  the
coefficient of $x^4$. Hence $Av.v=0$. \\

The       element       $A^2v=x(Av)=      -x^4+3x^3-2x^2+2x$       and
$A^2v.v=-1$. Similarly, $A^3v= -x^5+ 3x^4-2x^3+2x^2=$
\[=(-x^4+3x^3-2x^2+x-1)+3x^4-2x^3+2x^2=2x^4+x-1.\]                Hence
$A^3v.v=2$. Similarly, $A^4v=  2x^4-4x^3+5x^2-3x+2$ and $A^4v.v=2$. In
particular, $\e=  A^4v-v$ is isotropic. We now  compute the orthogonal
complement of $\e$. We have: 
\[\e.v= A^4v.v-v.v=0,\] 
\[\e.Av= A^4v.Av-v.Av=A^3v.v-Av.v=2-0=2,\] 
\[\e.   A^3v=   A^4v.A^3v-v.A^3v=   Av.v-A^3v.v=0-2=-2.\]  Hence   the
orthogonal     complement     of    $\e$     is     the    span     of
$\e,v,A^2v,v'=Av+A^3v$.  The  square  of  $\e '=Av+A^3v-v$  is  easily
computed to  be zero.  Hence the span  of $\e$  and $\e '$  is totally
isotropic and  $\Q$ rank of $V$  is two. \\  

Secondly, the  reflections $C_v, C_{A^2v}$ lie in  $\Gamma (\e)$.  The
reflection about $v'=Av+A^3v$ is computed to be
\[^{C_{Av}}(C_{A^3v})=
C_{C_{Av}(A^3v)}=C_{A^3v-(Av.A^3v)Av}=C_{A^3v+Av}=C_{v'},\]         and
therefore  also  lies  in  $\Gamma  (\e)$. By Lemma \ref{technical}, 
$\Gamma$  is arithmetic in $O(2,3)$; we have already shown that 
the $\Q$-rank is two.
\end{example}

\begin{example}
$f=(x-1)(x^2+1)^2$ and $g=(x+1) \frac{x^5-1}{x-1}$.  \\

In this  case, we show  that $\Q$  rank of $V$  is {\it one}  and that
$\Gamma$ is arithmetic. \\

We have $g=x^5+2x^4+2x^3+2x^2+2x+1$ and $f=x^5-x^4+2x^3-2x^2+x-1\equiv
0$. Hence  $Av=g-f= 3x^4+x+2$  and $Av.v=3$. Then  $A^2v= 3x^5+x^2+2x$
and   using    the   fact   that    $f\equiv   0$   we    get   $A^2v=
3(x^4-2x^3+2x^2-x+1)+x^2+2x= $ $  3x^4-6x^3 +7x^2-x+5$. Hence $A^2v.v=
3$.   Similarly  we   compute  $A^4v=   -2x^4+11x^3-6x^2+6x-3$.  Hence
$A^4v.v=-2$.  \\  

Therefore,     if    we    write     $\e=A^4v+v$    we     see    that
$\e^2=(A^4v)^2+v^2+2A^4v.v=  2+2+2(-2)=0$, and  $\e$ is  isotropic. We
compute the inner products $\e.A^kv$:
\[\e.v=A^4v.v+v.v=0,\e.A^4v=(A^4v)^2+v.A^4v)=2-2=0,\]
 \[\e.Av=A^4.Av+Av.v= A^3v.v+Av.v=-3+3=0, \]
\[\e.A^3v=A^4v.A^3v+v.A^3v=Av.v+A^3v.v=3-3=0.\]    Consequently,   the
orthogonal     complement    $\e     ^{\perp}$    is     spanned    by
$\e,v,Av,A^3v$. Since the  reflections about these vectors $v,Av,A^3v$
already lie  in $\Gamma (\e)$,  it follows from  Lemma \ref{technical}
that $\Gamma $ is arithmetic. \\

It remains to show  that the $\Q$-rank of $V$ is one.  Since $\e$ is a
rational isotropic vector, it is  enough to check that on the quotient
$\e  ^{\perp}/\Q\e$, the  restriction  of the  quadratic  form $h$  is
anisotropic over $\Q$ (it is isotropic over $\R$). With respect to the
basis $v,Av,A^3v$ of  $\e^{\perp}/\Q \e$, the matrix of  $h$ is of the
form
\[\begin{pmatrix} v^2 &  v,Av& v.A^3v \cr Av.v &  (Av)^2 & Av.A^3v \cr
A^3v.v & A^3v.Av & (A^3v)^2
\end{pmatrix}= \begin{pmatrix} \quad 2 & 3 & -3\cr \quad 3 & 2 & \quad
3 \cr -3 & 3 & \quad 2\end{pmatrix}.\] (We have used the invariance of
the dot product under $A$). This is the quadratic form $2Q$ where
\[Q=x^2+y^2+z^2+3xy-3xz+3yx.\] By completing the squares, and a linear
change of  variables, this form $Q$  can be shown to  be equivalent to
the quadratic form
\[5x^2-y^2+2z^2,\] which has no integral  zeros since $2$ is not a
quadratic residue modulo $5$. Hence  $\Q$-rank of $V$ is {\it one} but
the group $\Gamma$ is arithemtic.

\end{example}

\begin{example} $f=(x-1)(x^2+x+1)^2, \quad g= \frac{x^5-1}{x-1}(x+1)$. \\

We show that $\Gamma$ is arithmetic  with $\Q-rank (V)=2$.  Since $f$ has
{\it  two}  double  roots,  the   real  rank  of  $V$  is  two  (Lemma
\ref{realrank}).We have
\[f=x^5+x^4+x^3-x^2-x-1\equiv  0,  \quad  g=x^5+2x^4+2x^3+2x^2+2x+1.\]
Therefore,
\[Av=g-f=x^4+x^3+3x^2+3x+2,  \quad  {\rm  and} \quad  Av.v=1.\]  Using
$f\equiv 0$ we get
\[A^2v=   x^5+x^4+3x^3+3x^2+2=  2x^3+4x^2+3x+1,\quad{\rm   and}  \quad
A^2.v=0. \] Then
\[A^3v= 2x^4+4x^3+3x^2+x \quad{\rm and} \quad A^3v.v=2.\]
Finally,
\[A^4v=2x^4+x^3+3x^2+2x+2 \quad{\rm and} \quad A^4v.v=2.  \]
Put $\e=A^4v-v$; then $\e$ is isotropic. We compute 
\[\e.Av=A^3v.v-Av.v=2-1=1,  \quad  \e.A^3v= Av.v-A^3v.v=1-2=-1.\]  and
$\e.A^2v=0$.  Therefore,  $\e^{\perp}$  is  spanned  by  $\e,  v,A^2v,
A^3v+Av$  and hence by  the vectors  $\e,v,A^2v,v'=A^2v-A^3v-Av$. Note
that
\[^{C_{Av}C_{A^3v}}(C_{A^2v})=                ^{C_{Av}}(C_{A^2v-A^3v})=
C_{A^2v-Av-A^3v}=C_{v'}.\]  Then  by   Lemma  \ref{technical},  $\Gamma$  is
arithmetic.

Note also that if $w=A^3v+Av-v$, then 
\[w.w=(A^3v)^2+(Av)^2+v^2+2(A^3v.Av)-2(A^3.v)-2(Av.v)= \] 
\[=2+2+2+2.0-2.2-2.1=0\] and hence $w$ and $\e$ are orthogonal and are
both isotropic: $\Q-rank(V)=2$.
\end{example}

\begin{example}             $f=(x-1)(x^2+1)(x^2+x+1)             \quad
g=(x+1)(\frac{x^5-1}{x-1})$.  \\

The group  $G$  has $\R$  rank two  and
$\Q$-rank one. The group $\Gamma$ is arithmetic. \end{example}

\begin{example} $f=x^5-1, \quad g=(x+1)(x^2-x+1)^2$. \\

The group $G$ has
$\Q$-rank two and $\Gamma $ is arithmetic. \\

We    find     that    $v.Av=-1,v.A^2v=1,v.A^3v=1,v.A^4v=-1$.     Hence
$\e=A^2v+A^3v-v$ is isotropic and its orthogonal complement is generated
by  $\e, v,A^2v,A^4v-Av$. Since  $A^2v-A^4v+Av\in \e^{\perp}$,  and is
isotropic,    it     follows    that    $\Q-rank     (V)=2$.     Since
$^{C_{A^4v}}(C_{Av})=C_{Av-A^4v}$, it follows that the reflections
about these  basis elements $v,A^2v,A^2v-A^4v$  of $\e^{\perp}/\Q\e$ all
lie in  $\Gamma$ and  fix $\e$. Hence  (by arguments similar  to Lemma
\ref{technical}), we  get one non-trivial  element $u$ of  the integral
unipotent  radical $U(\Z)$ of  the parabolic  $P(\e)$ fixing  the line
through   $\e$  in   $\Gamma$.  To   be  specific,   the   element  is
$u=C_vC_{A^3v}C_vC_{A^2v}$, viewed as an element of $\Gamma (\e)$. \\

The conjugates of $u$ by  the reflections  in $\Gamma  (\e)$, generates a
finite  index subgroup  of $U(\Z)$  in $\Gamma  (\e)$.   Therefore, by
Theorem \ref{bamise}, $\Gamma$ is arithmetic.
\end{example}

\begin{example} $f=x^5-1,\quad g=(x+1)^3(x^2-x+1) $.  \\

The  group has  $\Q$-rank two  and $\Gamma  $ is  arithmetic.  We have
$Av.v=2,A^2v.v=1,A^3v.v=1,A^4v.v=2$.    Take   $\e=Av-1$.   Then   $\e
^{\perp}/\Q \e$ is the span of $v,A^3v,A^2v+A^4v$. It is also the span
of $v,A^3v,v'$ where $v'=A^2v+A^4v-2v$. Then
\[^{C_{A^2v}C_v}(C_{A^4v})=                   ^{C_{A^2v}}(C_{A^4v-2v})=
C_{A^4v-2v+A^2v}=C_{v'},\]  and lies  in $\Gamma  (\e)$;  the elements
$C_v,C_{A^3v}$ also do.  Hence, by  Lemma \ref{technical}, the group is
arithmetic. \\

The  element $\e'=  A^2v+A^4v-2v-A^3v$  lies in  $\e  ^{\perp}$ and  is
isotropic; hence $\Q-rank (G)=2$. \end{example}

\begin{example} $f= (x-1)(x^2+x+1)^2, \quad g= (x+1)(x^2-x+1)^2 $

The group has $\Q$-rank two and $\Gamma$ is arithmetic. \\

We    have    $Av.v=-2,A^2v.v=2,A^3v.v=2,    A^4v.v=-6$.   Take    $\e
=A^2v-v$. Then  $\e ^{\perp}$ is the span  of $v,Av,v'=A^4v+2A^3v$. Now
the reflections $C_v$ and  $C_{Av}$ lie in $\Gamma(\e)$; the following
computation
\[^{C_{A^3v}}(C_{A^4v})=C_{A^4v+2A^3v}=C_{v'}  \] shows  that  so does
$C_{v'}$. Hence by Lemma \ref{technical}, $\Gamma$ is arithmetic.
\end{example}

\begin{example} $ f=(x-1)(x^2+x+1)^2, \quad g=(x+1)(x^4-x^2+1)$. \\

The group has $\Q$-rank two and $\Gamma$ is arithmetic. \\

We  have $Av.v=0,  A^2v.v=2,A^3v.v=-2,A^4v.v=2$.  Take $\e  =A^2v-v$;
then $\e ^{\perp}$ is the span of $\e, v, Av, A^4v$ and hence by Lemma
\ref{technical}, the  group $\Gamma$ is arithmetic.  Since $A^4v-v$ is
perpendicular to $\e$ and is isotropic, the $\Q$-rank is two.

\end{example}

\begin{example} $f=x^5-1, \quad g= (x+1)(x^4-x^2+1) $. \\

The group $G$ has $\Q$-rank two and $\Gamma$ is arithmetic. \\

We  have  $Av.v=1,A^2v.v=-1,A^3v.v=-1,  A^4v.v=1$.   It  follows  that
$\e=A^2v-Av+v$  is  isotropic and  $\e  ^{\perp}$  is  spanned by  the
vectors $\e,v, Av, A^4v-A^3v$.  Since $\e'= A^2(\e)$ is also isotropic
and is  $(A^4v-A^3v)+A^2v$ it  follows that $\Q$-rank  of $G$  is two.
Moreover,  $A^4v-A^3v=  ^{C_{A^3v}}(  C_{A^4v})$  and hence  by  Lemma
\ref{technical}  (to check  that Lemma  \ref{technical}  applies, note
that    $\e=   A^2v-Av-v$    is   of    the   form    $g(v)-v$   where
$g=C_{Av}A^2=AC_vA$), $\Gamma$ is arithmetic.
\end{example}

\begin{example} $f=(x-1)(x^2+1)^2, \quad g=(x+1)(x^4-x^2+1)$. \\

$\Gamma$ is arithmetic and $\Q$-rank of $G$ is two. \\

A computation shows  that $Av.v=2,A^2v.v=-1, A^3v.v=-4,A^4v.v=2$. Take
$\e=A^4v -v$. Then
\[\e ^{\perp}/\Q  \e=<v,A^2v,v'= Av+A^3v=C_{A^3v}(Av)>. \]  The formula
for  $v'$   shows  from   Lemma  \ref{technical},  that   $\Gamma$  is
arithmetic. Moreover, if $\e'= v+A^2v-Av-A^3v=v+A^2v-v'$ then $\e'$ is
isotropic and orthogonal to $\e$; hence $\Q -rank (G)=2$.

\end{example}

\newpage
\noindent {\bf  Acknowledgements:} \\

I thank  Peter Sarnak for some very  helpful communications concerning
hypergeometric groups. I also thank Madhav Nori for suggesting 
(in a very different context) that Proposition \ref{bootstrap} should 
be true. \\

The  support of the  JC Bose  fellowship for  the period  2013-2018 is
gratefully acknowledged. \\

Most of the  computations were done while the author  was a visitor at
the ``Collaborative Research Centre, Spectral Structures and Topological
Methods'',  Department of  Mathematics, Bielefeld,  Germany, and TIFR Centre for Applicable Mathematics, Bangalore.   I thank
Herbert Abels and Mythily Ramaswamy for invitation to visit and the Bielefeld University and TIFR CAM, for hospitality. \\


\begin{thebibliography}{JPSH}


\bibitem{Ba-Mi-Se}  H.Bass,   J.Milnor  and  J-P.   Serre,
Solution of  the Congruence subgroup  problem for $SL_n $  ($n\geq 3$)
and $Sp_{2n}$ ($n\geq 2$), Publ. IHES, {\bf 33}, (1967), 59-137. \\

\bibitem{Beu-Hec} F.Beukers and G. Heckman, Monodromy for the
hypergeometric  function $_nF_{n-1}$, Invent.   Math. {\bf  95} (1989),
325-354. \\

\bibitem{Bor-Ti}   A.Borel  and  J.Tits,   Groupes  reductifs,
Publ. IHES.{\bf 27}, (1965), 55-150.\\

\bibitem{BT}  C.Brav and  H.Thomas,  Thin monodromy  in $Sp(4)$,  Math
Arxiv AG 1210.0523. \\

\bibitem{F} E.Fuchs,  The ubiquity of thin groups  (Proceedings of the
MSRI conference on thin groups and superstrong approximation, 2012). \\

\bibitem{FMS} E.Fuchs, C.Meiri  and P.Sarnak, Hypergeometric monodromy
groups for  the hypergeometric  equation and Cartan  involutions, Math
Arxiv GR 1305. 0729.\\

\bibitem{Ra1} M.S.Raghunathan, Discrete subgroups of Lie Groups, 
Springer Verlag. \\

\bibitem{Ra}   M.S.   Raghunathan,  A   note  on   generators  for
arithmetic   subgroups  of  algebraic   groups,  Pacific   Journal  of
Mathematics, {\bf 152} (1991), 365-373. \\

\bibitem{Sar}  P.  Sarnak, Notes  on thin  groups, Proceedings  of the
MSRI  conference  on thin  groups  and  superstrong approximation  hot
topics workshop, Feb 2012). \\

\bibitem{SV} S.Singh  and T.N.Venkataramana, Arithmeticity  of Certain
Symplectic  Hypergeometric  Groups  (to  appear in  Duke  Mathematical
Journal). \\

\bibitem{Ti}  J.   Tits,  Syst'eme`s  generateurs  de  groupes  de
congruence, C.R.Acad. Sci. Paris, Serie A {\bf 283} (1976), 693. \\

\bibitem{Va} L. Vaserstein,  The structure of classical arithmetic
groups  of rank  greater than  1 (English  Translation),  Math.  USSR,
Sbornik {\bf 20} (1973), 465-492. \\
 
\bibitem{Ve}  T.N.Venkataramana,  On  systems  of  generators  for
arithmetic subgroups  of higher rank groups, Pacific  Journal of Math.
{\bf 166} no.1 (1994), 193-212. \\





\end{thebibliography}
\end{document}